\documentclass [a4,10pt,twoside]{article}
\usepackage{amsmath, amssymb, amsthm, amsfonts, amsxtra, latexsym, amscd,
pb-diagram,  graphics,rotating,xy}

\theoremstyle{plain}
\newtheorem{dl}{Theorem}[section]
\newtheorem{bd}[dl]{Lemma}
\newtheorem{md}[dl]{Proposition}
\newtheorem{hq}[dl]{Corollary}
\newtheorem{dn}[dl]{Definition}
\newtheorem{nx}[dl]{Remark}

\usepackage[top=3.0cm, bottom=0cm, left=3cm,right=2cm] {geometry}

\newcommand{\I}{\mathcal{I}}

\newcommand{\C}{\mathcal{C}}

\begin{document}
\title{The factor set of Gr-categories of the type $(\Pi,A)$}  
\author{Nguyen Tien Quang}         
\maketitle

\begin{abstract}
Any $\Gamma$-graded categorical group is determined by a factor set of a categorical group. This paper studies the factor set of the group $\Gamma$ with coefficients in the categorical group of the type $(\Pi,A).$ Then, an interpretation of the notion of $\Gamma-$operator $3-$cocycle is presented and the proof of cohomological classification theorem for the a $\Gamma-$graded Gr-category is also presented.
\end{abstract}

\indent{\large\bf Introduction}

The notion of a graded monoidal category was presented by Fr\"{o}hlich and Wall [4] by generalization some manifolds of categories with the action of a group $\Gamma.$ Then, $\Gamma$ will be also regarded as a category with exactly one object, (say $\ast$), where the morphisms are the members of $\Gamma$ and the composition law is the group composition operation. A $\Gamma-${\it grading} on a category $\mathcal D$ is a functor $gr: \mathcal D\rightarrow \Gamma.$ The grading is called {\it stable} if for all $C\in ob\mathcal D, \sigma\in \Gamma,$ there is an equivalence $u$ in $\mathcal D$ with domain $C$ and $gr(u)=\sigma$. Then $\sigma$ is the {\it grade} of $u$. If $(\mathcal D, gr)$ is a $\Gamma-$graded category, we define $Ker\mathcal D$ to be the subcategory consisting of all morphisms of grade 1. 

A $\Gamma-${\it monoidal category} consists of: a stably $\Gamma-$graded category $(\mathcal D, gr)$,  $\Gamma-$functors
$$
\otimes: \mathcal D \times_{\Gamma} \mathcal D \longrightarrow \mathcal D\ ,\ \ \ \I: \Gamma \longrightarrow \mathcal D
$$
and natural equivalences (of grade 1)  $a_{X,Y,Z} : (X\otimes Y)\otimes Z  \stackrel{\sim}{\longrightarrow}X\otimes (Y\otimes Z ),$
$l_X : I\otimes X  \stackrel{\sim}{\longrightarrow}X,$
 $r_X : X\otimes I  \stackrel{\sim}{\longrightarrow}X,$
where $I = I(\ast)$,
 satisfying coherence conditions of a monoidal category.
 
For any $\Gamma-$graded category $(\mathcal D, gr),$ authors wrote $Rep(\mathcal D, gr)$ for the category of $\Gamma-$fuctors $F:\Gamma\rightarrow \mathcal D$, and natural transformations. An object of $Rep(\mathcal D, gr)$ thus consists an object $C$ of  $\mathcal D$ with homomorphism $\Gamma\rightarrow Aut_{\mathcal D}(C).$ Homomorphism $\Gamma\rightarrow Aut_{\mathcal D}(I)$ is the right inverse of the graded homomorphism $Aut_{\mathcal D}(I)\rightarrow \Gamma.$ In other words,  $Aut_{\mathcal D}(I)$ is a split extension of the normal subgroup $N$ of automorphisms of grade 1 by the subgroup which is isomorphic to $\Gamma.$ The extension defines an action of $\Gamma$ on $N$ by
$$\sigma u=I(\sigma)\circ u\circ I(\sigma^{-1}).$$

In [1], authors considered the $\Gamma-$graded extension problem of categories as a categorization of the group extension problem. The groups $A, B$ in the short exact sequence:
$$0\rightarrow A\rightarrow B\rightarrow \Gamma\rightarrow 1$$
are replaced with the categories $\C, \mathcal D$. A $\Gamma-$monoidal extension of the monoidal category $\C$ is a $\Gamma-$monoidal category $\mathcal D$ with a monoidal isomorphism $j: \C \longrightarrow Ker \mathcal D.$

The construction and classification problems of $\Gamma-$monoidal extension were solved by raising the main results of Schreier-Eilenberg-MacLane on group extensions to categorical level. With the notations of {\it factor set} and {\it crossed product extension,} authors proved that there exists a bijection
$$\Delta: H^{2}(\Gamma, \C)\leftrightarrow Ext(\Gamma, \C)$$
between the set of congruence classes of factor sets on $\Gamma$ with coefficients in the monoidal category $\C$ and the set of congruence classes of $\Gamma-$extensions of $\C$.

The case $\C$ is a categorical group (also called a Gr-category) was considered in [2]. Then, the $\Gamma-$equivariant structure appears on $\Pi-$module $A$, where $\Pi=\Pi_0(\C), A=Aut_{\C}(1)=\Pi_1(C).$ $\Gamma-$extensions and $\Gamma-$functors are classified by functors
$$ch: _{\Gamma}\C\mathcal G\rightarrow \mathcal H^{3}_{\Gamma}\ ;\ \int_{\Gamma}: \mathcal Z^{3}_{\Gamma}\rightarrow _{\Gamma}\C\mathcal G,$$
\noindent where $_{\Gamma}\C\mathcal G$ is the category of $\Gamma-$extensions, $\mathcal Z^{3}_{\Gamma}$ is the category in which any object is a triple $(\Pi, A, h)$, where $(\Pi, A)$ is a $\Gamma-$pair and $h\in H^{3}$; $\mathcal H^{3}_{\Gamma}$ is the category obtained from $\mathcal Z^{3}_{\Gamma}$ when $h\in Z^{3}_{\Gamma}(\Pi, A)$ is replaced with $h\in H^{3}_{\Gamma}(\Pi, A).$ 

As we know, each categorical group is equivalent to a categorical group of the type $(\Pi,A)$ and the unit constraint is strict (in the sense $l_x=r_x=id_x$). So we may solve the classification problem for this special case thanks to the desription of Gr-functors of Gr-categories of the type $(\Pi,A)$ [6]. We may better describe the factor set, and show that the $\Gamma-$equivariant structure of $A$ is a necessary condition of the factor set. Thus, we may construct $\Gamma-$operator $3-$cocycles as a induced version of a factor set, instead of using complex construction as in [2]. By this way, we may obtain the classification theorem in a stronger form than the result in [2], that is the bijection: 
$$\Omega: \mathcal S_{\Gamma}(\Pi, A)\rightarrow H^3_{\Gamma}(\Pi,A)$$
where $\mathcal S_{\Gamma}(\Pi, A)$ is the set of congruence classes of $\Gamma-$extensions of Gr-categories of the type $(\Pi,A)$.
 
The classification problem of $\Gamma-$functors follows this method will be presented in another paper.
\section {Some notions}
\indent Let $\Pi$ be a group and $A$ be a left $\Pi-$module. A Gr-category of the type $(\Pi,A)$ is a category $\mathcal S=\mathcal S(\Pi,A, \xi)$ in which objects are elements $ x\in\Pi,$ and morphisms are automorphisms
   $$Aut(x)=\{x\}\times A.$$
The composition of two morphisms is defined by
$$(x,u)\circ (x,v)=(x, u+v)$$

\noindent The operation $\otimes$ is defined by
$$ x \otimes y = xy  $$
$$ (x,u) \otimes (y,v) = (xy,u+xv).$$
The associative constraint $a_{x,y,z}$ is a normalized $3-$cocycle (in the sense of group cohomology) $\xi\in Z^{3}(\Pi,A)$, and the unit constraint is strict. Then from now on, a Gr-category of the type $(\Pi, A)$ refers to the one with above properties.
\begin{dn} Let $\Gamma$ be a group and let $(\C, \otimes)$ be any monoidal category. We say that a factor set on $\Gamma$ with coefficients in $(\C, \otimes)$ is a pair $(\theta,F)$ consisting of: a family of monoidal autoequivalences
\[F^\sigma = (F^\sigma,\widetilde{F^\sigma},\widehat{F^\sigma}):  \C\longrightarrow \C , (\sigma \in \Gamma)\] 
and a family of isomorphisms of monoidal functors
\[\theta^{\sigma,\tau} : F^\sigma F^\tau \stackrel{\sim}{\longrightarrow} F^{\sigma\tau} , (\sigma, \tau \in \Gamma)\] 
satisfying the conditions\\
\indent i) $F^1 = id_{\C}$\\
\indent ii) $ \theta^{1,\sigma} = id_{F^\sigma} = \theta^{\sigma,1}$ ($\sigma \in \Gamma$)\\
\indent iii) for all $\forall \sigma,\tau,\gamma \in \Gamma,$ the following diagrams are commutative
\[\begin{CD}
F^{\sigma} F^{\tau} F^{\gamma} @> \theta^{\sigma,\tau}F^{\gamma} >> F^{\sigma\tau}F^{\gamma}\\
@V F^{\sigma}\theta^{\tau,\gamma} VV           @VV \theta^{\sigma\tau,\gamma} V\\
F^{\sigma}F^{\tau\gamma} @> \theta^{\sigma,\tau\gamma} >> F^{\sigma\tau\gamma}
\end{CD}\]
\end{dn}
In [6], authors described monoidal functors between monoidal categories of the type $(\Pi, A).$ Thanks to this description, we will prove the necessary conditions of a factor set.

\begin{dn}{[6]}
Let $\mathcal S=(\Pi, A,\xi),\ \mathcal S'=(\Pi',A',\xi')$ be Gr-categories. A functor $F: \mathcal S\to \mathcal S'$ is called a functor of the type $(\varphi,f)$ if
$$F(x)=\varphi(x)\ \ ,\ \ F(x,u)=(\varphi(x), f(u)) $$
and $\varphi:\Pi\to \Pi'$, $f:A\to A'$ is a pair of group homomorphisms satisfying $f(xa)=\varphi (x)f(a)$ for $x\in \Pi, a\in A.$
\end{dn}
We have

\begin{dl}{ [6]}
Let $\mathcal S=(\Pi, A, \xi),\ \mathcal S'=(\Pi',A',\xi')$ be
Gr-categories and $F=(F,\widetilde{F}, \widehat{F})$ be a Gr-functor from $\mathcal
S$ to $\mathcal S'$. Then, $F$ is a functor of the type $(\varphi,f).$
\end{dl}
According to this theorem, any monoidal autoequivalence $F^\sigma$ is of the form $F^\sigma = (\varphi^\sigma,f^\sigma)$. This remark is used frequently throughout this paper.
\begin{dn} {[2]} Let $\Gamma$ be a group, $\Pi$ be a $\Gamma-$group. A $\Gamma-$module $A$ is a equivariant module on $\Gamma-$group $\Pi$ if $A$ is a $\Pi-$module satisfying
$$\sigma (xa)=(\sigma x)(\sigma a),$$
for all $\sigma \in \Gamma, x\in \Pi$ and $a \in A$.
\end{dn}
\section {$\Gamma-$graded extension of a $Gr-$category of the type $(\Pi,A)$} 
\indent For a given factor set $(\theta, F)$, we may construct a $\Gamma-$graded crossed product extension of $\C$, denoted by $\Delta(\theta, F)$ as follows:\\
$$ \begin{CD} \mathcal{C} @>j>>\Delta(\theta,F)@>g>>G \end{CD}$$
where $\Delta(\theta, F)$ is a category in which objects are objects of $\C$ and morphisms are pairs $(u, \sigma): A\rightarrow B$ where $\sigma\in\Gamma$ and $u: F^\sigma(A)\rightarrow B$ is a morphism in $\C$. The composition of two morphisms:
$$ \begin{CD} A@>(u, \sigma)>>B@>(v, \tau)>>C \end{CD}$$
is defined by:\\
$$(u, \sigma)\cdot(v, \tau) = (u\cdot F^\sigma(v)\cdot\theta^{\sigma, \tau}(A)^{-1}, \sigma\tau).$$
This composition is associative and the unit exists thanks to cocycle and normalized conditions i), ii), iii) of $(\theta, F)$.\\
\indent A stably $\Gamma-$grading on $\Delta(\theta, F), g:
\Delta(\theta, F)\rightarrow \Gamma$ is defined by $g(u, \sigma) =\sigma$, and the bijection $\begin{CD}j: \mathcal{C}@>\sim>>
Ker(\Delta(\alpha, F))\end{CD}$ is defined by:
 $$\begin{CD} j(A @>u>> B) = (A @>(u, 1)>> B)
\end{CD}$$

\begin{md} If $G:\C\rightarrow \C'$ and $H:C'\rightarrow \C$ are monoidal equivalence such that $\alpha:G\circ H\cong id_{\C'},$ and $\beta:H\circ G\cong id_{\C},$ and $\mathcal D$ is a crossed product $\Gamma-$extension of $\C$ by the factor set $(\theta, F),$ then the quadruple $(G, H,\alpha,\beta)$ induces:\\
 \indent i) the factor set $(\theta', F')$ of $\C'$,\\
 \indent ii) a $\Gamma-$ equivalence
 $$\Delta(\theta, F))\leftrightarrow \Delta(\theta', F')).$$ 
 \end{md}
\begin{proof}
\indent  i) Let $F'^\sigma$ be the composition $H\circ F^\sigma\circ G$ and $$\theta'^{\sigma, \tau}_{X}=G(\theta^{\sigma, \tau}_{HX}\circ F^\sigma(\beta_{F^{\tau}HX})).$$
One can verify that $(\theta', F')$ is a factor set of $\C'$.

ii) We extend the functor $G:\C\rightarrow \C'$ to a functor
$$\widetilde{G}:\Delta(\theta, F))\leftrightarrow \Delta(\theta', F'))$$
as follows: for the object $X$ of $\C,$ let $\widetilde{G}X=GX;$ for the morphism $(u, \sigma): X\rightarrow Y,$ where $u:F^{\sigma}X\rightarrow Y,$ let $$\widetilde{G}(u, \sigma)=(G(u\circ F^{\sigma}(\beta_{X})), \sigma).$$
One can verify that $\widetilde{G}$ is a $\Gamma-$equivalentce
\end{proof}
From the above proposition, it is deduced that
\begin{hq}
Any $\Gamma-$extension of a Gr-category is equivariant to a $\Gamma-$extension of a Gr-category of the type $(\Pi, A).$
\end{hq}

We now prove some necessary conditions for the existence of a factor set.
 
\begin{dl} \label{dl2.3} Let $\Gamma$ be a group and $\mathcal S=\mathcal S(\Pi, A,\xi)$ be a Gr-category. If $(\theta, F)$ is a factor set of $\Gamma$, with coefficients in $\mathcal S,$ then:\\
 \indent i) there exists a group homomorphism
 $$\varphi: \Gamma \longrightarrow Aut \Pi\ ;\  f:\Gamma \longrightarrow AutA, $$ 
\noindent and $A$ is equiped with a $\Pi-$module $\Gamma-$equivariant structure, induced by $\varphi, f$,\\
\indent ii) in Definition 2.1, the condition i) of a factor set can be deduced from the remaining conditions.
\end{dl}
\begin{proof}
\indent i) According to Theorem 1.1, any autoequivalence $F^{\sigma},\sigma \in \Gamma, $ of a factor set is of the form
$$(\varphi^{\sigma}, f^{\sigma}): \mathcal S \longrightarrow \mathcal S.$$
Since $\widetilde{F^\sigma}_{x,y} : F^{\sigma}(xy) \longrightarrow F^{\sigma}x.F^{\sigma}y , \sigma \in \Gamma$ is a morphism in $(\Pi,A),$ we have
\[F^{\sigma}(xy) = F^{\sigma}x.F^{\sigma}y, \forall x,y \in \Pi.\]
This stated that $\varphi^{\sigma}=F^{\sigma}$ is a endomorphism in $\Pi$. Furthermore, $F^{\sigma}$ is an equivalence, so that $\varphi^{\sigma}$ is an automorphism of group $\Pi$, that is $\varphi^{\sigma} \in Aut\Pi$. On the other hand, since $\theta^{\sigma,\tau}_x : F^{\sigma}F^{\tau}x \longrightarrow F^{\sigma\tau}x$ is an arrow in $(\Pi,A),$ we have
\[(F^{\sigma}F^{\tau})(x) = F^{\sigma\tau}(x) , \forall x\in \Pi. \]
Thus, $\varphi^{\sigma}\varphi^{\tau} = \varphi^{\sigma\tau}$. This proved that
\[\begin{matrix}
\varphi: \Gamma &\longrightarrow& Aut\Pi\\
\sigma &\longmapsto& \varphi^{\sigma}
\end{matrix}\]
is a homomorphism of groups. Then $\varphi^1=\varphi(1)=id_{\Pi }$.

 Let $\widetilde{F^\sigma}_{x,y} = (\sigma(xy), \widetilde{f}^{\sigma}(x,y))$, in which $\sigma(xy)=\varphi^{\sigma}(xy); \widetilde{f}^{\sigma}: \Pi^{2}\longrightarrow A$ and
$\widehat{F^\sigma}= (1,c^{\sigma}): F^{\sigma}1 \longrightarrow 1$ are maps.

From the definition of the monoidal functor $F^\sigma$, we have
\begin{equation}
\sigma x.\widetilde{f}^{\sigma}(y,z)-\widetilde{f}^{\sigma}(xy,z)+
\widetilde{f}^{\sigma}(x,yz)-\widetilde{f}^{\sigma}(x,y)=a(\sigma x,\sigma y,\sigma z)-f^{\sigma}(a(x,y,z))
\end{equation}
\begin{equation}
 (\varphi^{\sigma}x)c^{\sigma} + \widetilde{f}^{\sigma}(x,1)=0
 \end{equation}
\begin{equation}
 c^{\sigma} + \widetilde{f}^{\sigma}(1,x)=0
 \end{equation}
 We now observe isomorphisms of monoidal functors
$\theta^{\sigma,\tau}  = (\theta^{\sigma,\tau}_x)$ where $$(\theta^{\sigma,\tau}_x) = (\varphi^{\sigma}x, t^{\sigma,\tau}(x)) : F^{\sigma}F^{\tau}(x) \longrightarrow F^{\sigma\tau}(x),$$
in which $t^{\sigma,\tau} : \Pi \longrightarrow A$ are maps.\\
We have the following commutative diagrams
\[\begin{CD}
F^{\sigma}F^{\tau}(x) @> (\bullet, t^{\sigma,\tau}(x)) >> F^{\sigma\tau}(x)\\
@V (\bullet, f^{\tau}f^{\sigma}(a)) VV           @VV (\bullet,  f^{\sigma\tau}(a)) V\\
F^{\sigma}F^{\tau}(x) @> (\bullet, t^{\sigma,\tau}(x)) >> F^{\sigma\tau}(x) 
\end{CD}\]
\[\begin{CD}
F^{\sigma}F^{\tau}(xy) @> (\bullet, \widetilde{f}^{\sigma}(F^{\tau}x, F^{\tau}y)+f^{\sigma}\widetilde{f}^{\tau}(x,y)) >> F^{\sigma}F^{\tau}(x).F^{\sigma}F^{\tau}(y)\\
@V (\bullet, t^{\sigma,\tau}(xy)) VV           @VV (F^{\sigma\tau}x,t^{\sigma,\tau}(x))\otimes (F^{\sigma}y,t^{\sigma,\tau}(y)) V\\
F^{\sigma\tau}(xy) @> (\bullet, \widetilde{f}^{\sigma\tau}(x, y)) >> F^{\sigma\tau}(x).F^{\sigma\tau}(y)\\
\end{CD}\]
\[\begin{diagram}
\node{F^{\sigma}F^{\tau}1} \arrow[2]{s,l}{(1,t^{\sigma,\tau}(1))} \arrow{se,t}{(1,f^{\sigma}(c^{\tau})+c^{\sigma})}\\
\node[2]{1}\\
\node{F^{\sigma\tau}1} \arrow{ne,b}{(1,c^{\sigma\tau})}
\end{diagram}\]
\noindent From which we are led to
\begin{equation}
f^{\sigma}f^{\tau} = f^{\sigma\tau}
\end{equation}
\begin{equation}
F^{\sigma\tau}(x)t^{\sigma,\tau} (y) - t^{\sigma,\tau}(xy)+ t^{\sigma,\tau}(x) = \widetilde{f}^{\sigma\tau}(x,y)-\widetilde{f}^{\sigma}(F^{\tau}_x,F^{\tau}_y)- f^{\sigma}(\widetilde{f}^{\tau}(x,y)) 
\end{equation}
\begin{equation}
f^{\tau}(c^{\tau})+c^{\tau}- c^{\sigma \tau}=t^{\sigma, \tau}(1)
\end{equation}
From the equality (4), we are led to a homomorphism $$f: \Gamma \longrightarrow AutA$$ given by $f(\sigma)=f^{\sigma}$ and so that $f^1 = f(1) = id_A$.

 Now, let 
 \begin{equation} \sigma x = \varphi^{\sigma}x, \quad
\sigma c = f^{\sigma}c 
\end{equation}
 for all  $\sigma \in \Gamma, x \in \Pi
, c\in A.$
Thus, since $F^{\sigma}$ is a functor of the type $(\varphi^{\sigma}, f^{\sigma}),$ $\varphi^{\sigma}(xb)=\varphi^{\sigma}(x)f^{\sigma}(b)$, or
$$\sigma(xb)=\sigma(x)\sigma(b),$$
that is $A$ is a $\Pi$-module $\Gamma$-equivariant.
 
\indent ii) From the condition ii) in the definition of a factor set, we are led to
$$t^{\sigma,1} = t^{1,\sigma} = 0 $$
From the equality (5), for $\tau =1$, we obtain $\widetilde{f}^{1}_{x,y} = 0$, that is $\widetilde{F^1}_{x,y} = id$. From the equality (3), for $\sigma =1$, we have: $c^1 =0$, that is $\widehat{F}^1 = id$. Thus, $F^1$ is an identity monoidal functor. The theorem is proved.
\end{proof}



\section{Enough strict factor set and induced 3-cocycle}
When the monoidal category $\C$ is replaced with the
categorical group of the type $(\Pi,A)$ we obtain better
descriptions than in the general case. For example, in Theorem 3.2,
authors proved that the condition i) of the definition of factor
set of categorical groups of the type $(\Pi,A)$ is redundant. Now,
we continue "reducing" this concept in terms of other face.

In this paper, we call a factor set $(\theta,F)$ {\it \textbf{enough strict}} if $\widehat{F^{\sigma}}= id_I$ for all $\sigma \in \Gamma$.

\begin{dn} Let $\Gamma$ be a group and $\mathcal{C}$ be a Gr-category of the type $(\Pi,A)$. Factor sets $(\theta,F)$ and $(\mu,G)$ on $\Gamma$ with coefficients in $\mathcal{C}$ are cohomologous if there exists a family of isomorphisms of monoidal functors
\[
u^{\sigma}: (F^{\sigma}, \widetilde{F^{\sigma}},\widehat{F^{\sigma}})
\stackrel{\sim}{\longrightarrow} (G^{\sigma},
\widetilde{G^{\sigma}},\widehat{G^{\sigma}}) \quad (\sigma \in \Gamma)
\]  
satisfying
\[u^1 = id_{(\Pi,A)}\]
\[u^{\sigma\tau}.\theta^{\sigma\tau}=\mu^{\sigma,\tau}.u^{\sigma}G^{\tau}.F^{\sigma}u^{\tau} \quad(\sigma,\tau \in \Gamma)\]
\end{dn}

\begin{nx}\label{nx1} If the two representatives $(\theta, F), (\mu, G)$ are cohomologous, then $F^{\sigma}=G^{\sigma}, \sigma\in \Gamma.$ 
\end{nx}

Indeed, from the definition of cohomologous factor sets, there exists a family of isomorphisms of monoidal functors $$u^{\sigma}:
(F^{\sigma}, \widetilde{F^{\sigma}},
\widehat{F}^{\sigma})\rightarrow (G^{\sigma}, \widetilde{G^{\sigma}}
\widehat{G}^{\sigma})\quad (\sigma \in \Gamma).$$

Since$u_x^{\sigma}: F^{\sigma}x \rightarrow G^{\sigma}x $ is an arrow in $(\Pi,A),$ we have
$G^{\sigma}x = F^{\sigma}x.$

Furthermore, for any $a \in A,$ 
by the commutativity of the diagram \[ \begin{diagram} \node{F^{\sigma}x}
\arrow{e,t}{u_x^{\sigma}} \arrow{s,l}{F^{\sigma}(x,a)}
\node{G^{\sigma}x} \arrow{s,r}{G^{\sigma}(x,a)}\\
\node{F^{\sigma}x}
\arrow{e,t}{u_x^{\sigma}}\node{G^{\sigma}x}\end{diagram}
\]by the commutativity of the diagram $F^{\sigma}(x,a)=G^{\sigma}(x,a).$

Extending Lemma 1.1 [2] for a factor set, we have

\begin{bd} Let $\mathcal S$ be a categorical group of the type
$(\Pi,A)$. Any factor set $(\theta,F)$ on $\Gamma$ with cofficients
in $\mathcal S$ is cohomologous to an enough strict factor set $(\mu,
G)$.
\end{bd}
\begin{proof}
For each $\sigma \in \Gamma$, consider a family of isomorphisms in $\mathcal S$: 
$$u_x^{\sigma} =
\begin{cases}
id_{F^{\sigma}x} & \text{ if $x \neq 1$}\\
(\widehat{F^{\sigma}})^{-1} & \text{ if $x = 1$}
\end{cases}$$

\noindent where $1\neq\sigma \in \Gamma$, and  $ u^1=id.$

Then, we define $G^{\sigma}$ in a unique way such that $u^{\sigma}: G^{\sigma} \rightarrow F^{\sigma}$ is a natural transformation by setting $G^{\sigma}=F^{\sigma}$ and:
$$ \widetilde{G^{\sigma}}_{x, y}= (u^{\sigma}_{x} \otimes
u^{\sigma}_{y})^ {-1}\widetilde{F^{\sigma}}_{x, y}(u^{\sigma}_{xy}); \quad
\widehat{G^{\sigma}}= id_I$$
For such setting, clearly we have \[ G^{\sigma}=(G^{\sigma},\widetilde{G^{\sigma}},\widehat{G^{\sigma}}):
\mathcal S \rightarrow \mathcal S \] is a monoidal
equivalence. In particular, we have:
$ \widehat{G^{\sigma}}= id_I$.
This states the enough strictness of the family of functors $G^{\sigma}$, as well as of the factor set $(\mu,G)$.

Now, we set $\mu^{\sigma,\tau}: G^\sigma G^\tau \rightarrow G^{\sigma \tau}$ the natural transformation which makes the following diagram
\[\begin{diagram}
\node{G^\sigma G^\tau} \arrow[2]{e,t,..}{\mu^{\sigma,\tau}} \arrow {se,r}{G^\sigma u^\tau}
\node[2]{G^{\sigma\tau}} \arrow[2]{e,t}{u^{\sigma\tau}}
\node[2]{F^{\sigma\tau}}\\
\node[2]{G^\sigma F^\tau}\arrow[2]{e,t}{u^\sigma F^\tau}
\node[2]{F^\sigma F^\tau} \arrow{ne,r}{\theta^{\sigma,\tau}}\\
\node[3]{\textrm{Diagram 1}}
\end{diagram}\]

commute, for all $\sigma$, $\tau \in \Gamma$. Clearly, $\mu^{\sigma,\tau}$ is a
isomorphism of monoidal functors.

We will prove that the family of $\mu^{\sigma, \tau}$ satisfy the
condition ii) of the definition of a factor set. 
We now prove that they satisfy the condition iii). Consider the diagram:

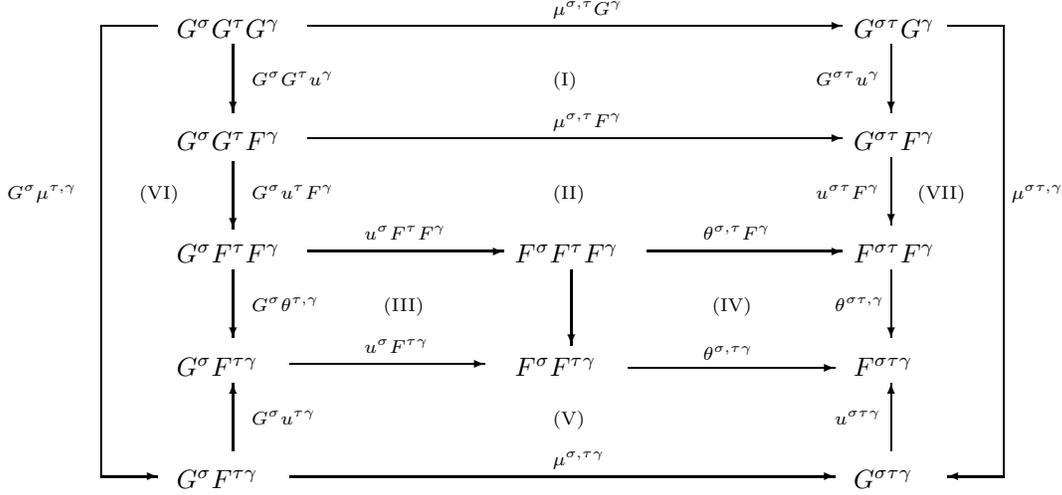
\[\setlength\unitlength{0.5cm}
\begin{picture}(19,13)
\put(0, 13){$G^\sigma G^\tau G^\gamma$}
\put(0, 10){$G^\sigma G^\tau F^\gamma$}
\put(0, 7){$G^\sigma F^\tau F^\gamma$}
\put(0, 4){$G^\sigma F^{\tau \gamma}$}
\put(0, 1){$G^\sigma F^{\tau \gamma}$}

\put(9, 7){$F^\sigma F^\tau F^\gamma$}
\put(9, 4){$F^\sigma F^{\tau \gamma}$}

\put(18, 13){$G^{\sigma \tau} G^\gamma$}
\put(18, 10){$G^{\sigma \tau} F^\gamma$}
\put(18, 7){$F^{\sigma \tau} F^\gamma$}
\put(18, 4){$F^{\sigma \tau \gamma}$}
\put(18, 1){$G^{\sigma \tau \gamma}$}

\put(3.5, 13.3){\vector(1, 0){14.2}} \put(10, 13.6){\scriptsize $\mu^{\sigma, \tau} G^\gamma$}
\put(3.5, 10.3){\vector(1, 0){14.2}}\put(10, 10.6){\scriptsize $\mu^{\sigma, \tau} F^\gamma$}
\put(3.5, 7.3){\vector(1, 0){5.2}}\put(5, 7.6){\scriptsize $u^\sigma F^\tau F^\gamma$} \put(12.5, 7.3){\vector(1, 0){5.2}}\put(14, 7.6){\scriptsize $\theta ^{\sigma ,\tau}F^\gamma$}
\put(3, 4.3){\vector(1, 0){5.2}}\put(5, 4.6){\scriptsize $u^\sigma F^{\tau \gamma}$} \put(12, 4.2){\vector(1, 0){5.5}}\put(14, 4.4){\scriptsize $\theta ^{\sigma ,\tau\gamma}$}
\put(3, 1.3){\vector(1, 0){14.5}} \put(10, 1.6){\scriptsize $\mu^{\sigma, \tau \gamma}$}

\put(1.5, 12.8){\vector(0,-1){1.8}}\put(2.0,11.7){\scriptsize $ G^\sigma G^\tau u^\gamma$}
\put(1.5, 9.7){\vector(0, -1){1.8}}\put(2.0,8.7){\scriptsize $G^\sigma u^\tau F^\gamma$}
\put(1.5, 6.8){\vector(0, -1){1.8}}\put(2.0,5.7){\scriptsize $G^\sigma \theta ^{\tau,\gamma}$}
\put(1.5, 2){\vector(0, 1){1.8}}\put(2.0,2.7){\scriptsize $G^\sigma u^{\tau\gamma}$}

\put(19.0, 12.8){\vector(0, -1){1.8}}\put(17,11.7){\scriptsize $G^{\sigma \tau} u^\gamma$}
\put(19.0, 9.8){\vector(0, -1){1.8}}\put(17,8.7){\scriptsize $u^{\sigma \tau}F^\gamma$}
\put(19.0, 6.8){\vector(0, -1){1.8}}\put(17.5,5.7){\scriptsize $\theta ^{\sigma \tau,\gamma}$}
\put(19.0, 2){\vector(0, 1){1.8}}\put(17.5,2.7){\scriptsize $u^{\sigma \tau \gamma}$}

\put(10.5, 6.8){\vector(0, -1){2.0}}

\put(-0.5,13.3){\line(-1,0){1.5}}
\put(-2,13.3){\line(0,-1){12}}
\put(-2,1.3){\vector(1,0){1.5}}

\put(20.5,13.3){\line(1,0){1.5}}
\put(22,13.3){\line(0,-1){12}}
\put(22,1.3){\vector(-1,0){1.5}}

\put(10, 11.7){\scriptsize (I)}
\put(10, 8.7){\scriptsize (II)}
\put(5.5, 5.7){\scriptsize (III)}
\put(14.0, 5.7){ \scriptsize (IV)}
\put(10, 2.7){\scriptsize (V)}
\put(-1.0,8.7){\scriptsize (VI)}
\put(19.7,8.7){\scriptsize (VII)}

\put(-4.5, 8.7){\scriptsize $G^\sigma \mu^{\tau, \gamma}$}
\put(22.2, 8.7){\scriptsize $\mu^{\sigma \tau, \gamma}$}

\end{picture}\]
In this diagram, the region (I) commutes thanks to the naturality of $\mu^{\sigma,\tau};$ the regions (II), (V), (VI), (VII) commute thanks to the Diagram 1; the region (III) commutes thanks to the naturality of $u^\sigma;$ the region (IV) commutes thanks to the definition of the factor set $(\theta, F).$ So the perimater commutes. This completes the proof.
\end{proof}
We now show that any factor set induces a $\Gamma-$operator $3-$cocycle based on the following definition.

\indent Let a $\Gamma-$pair $(\Pi,A)$, that is, a $\Pi-$module $\Gamma-$equavariant $A$,
cohomology groups
$H_{\Gamma}^n(\Pi,A)$ studied in [3]. We recall that cohomology group
$H_{\Gamma}^n(\Pi,A)$,  with $n \leq 3$, can be computed as the cohomology group of the struncated
cochain complex:
\[\begin{CD}
 \stackrel{\sim}{C}_{\Gamma}(\Pi,A):0@>>> C_{\Gamma}^1(\Pi,A)
@>\partial>>C_{\Gamma}^2(\Pi,A)@>\partial>>Z_{\Gamma}^3(\Pi,A)@>>>0,
\end{CD}\]
\noindent in which $C_{\Gamma}^1(\Pi,A)$ consists of normalized maps $f:\Pi\rightarrow A$, $C_{\Gamma}^2(\Pi,A)$
consists of normalized maps $g:\Pi^2\cup(\Pi \times \Gamma)\rightarrow A$ and $Z_{\Gamma}^3(\Pi,A)$ consists of normalized maps $h:\Pi^3\cup (\Pi^2\times \Gamma)\cup (\Pi \times \Gamma^2)\rightarrow A$ satisfying the following $3-$cocycle conditions:\\
\begin{equation}
h(x,y,zt)+h(xy,z,t)=\ x(h(y,z,t))+h(x,yz,t)+h(x,y,z)
\end{equation}
\begin{equation}
{\sigma}(h(x,y,z))+h(xy,z,\sigma)+h(x,y,\sigma)=h(\sigma(x),
\sigma(y),\sigma(z))+(\sigma(x))(h(y,z,\sigma))+h(x,yz,\sigma)
\end{equation}
\begin{equation}
\sigma(h(x,y,\tau))+h(\tau(x),\tau(y),\sigma)+h(x,\sigma,\tau)+(\sigma\tau(x))(h(y,\sigma,\tau))=h(x,y,\sigma\tau)+h(xy,\sigma,\tau)
\end{equation}
\begin{equation}
\sigma(h(x,\tau,\gamma))+h(x,\sigma,\tau\gamma)=h(x,\sigma\tau,\gamma)+h(\gamma(x),\sigma,\tau)
\end{equation}
for all $x,y,z,t\in \Pi;\ \sigma,\tau,\gamma\in \Gamma$.\\
 For each $f\in C_{\Gamma}^1(\Pi,A)$, the coboundary $\partial f$ is given by
\begin{equation}
(\partial f)(x,y)=x(f(y))-f(xy)+f(y),
\end{equation}
\begin{equation}
(\partial f)(x,\sigma)=\sigma(f(x))-f(\sigma(x)),
\end{equation}
and for each $g\in C_{\Gamma}^2(\Pi,A), \partial g$ is given by:
\begin{equation}
(\partial g)(x,y,z)=x(g(y,z))-g(xy,z)+g(x,yz)-g(x,y),
\end{equation}
\begin{equation}
(\partial
g)(x,y,\sigma)=\sigma(g(x,y))-g(\sigma(x),\sigma(y))-\sigma(x)(g(y,\sigma))+g(xy,\sigma)-
g(x,\sigma),
\end{equation}
\begin{equation}
(\partial
g)(x,\sigma,\tau)=\sigma(g(x,\tau))-g(x,\sigma\tau)+g(\tau(x),\sigma).
\end{equation}
\begin{md} Any enough strict factor set $(\theta,F)$ on $\Gamma$ with coefficients in
 $\mathcal{S}= (\Pi,A,\xi)$ induces an element $h \in Z^3_{\Gamma}(\Pi,A)$.
\end{md}
 
\begin{proof}
 Suppose
$F^{\sigma}=(F^{\sigma},\widetilde{F^\sigma},id)$. Then, we can write
$$
\widetilde{F^\sigma}_{x,y}=(\varphi^{\sigma}(xy),\widetilde{f}^{\sigma}(x,y))=(\sigma(xy),g(x,y,\sigma))$$

\noindent where $ \widetilde{f}: \Pi^2\times \Gamma \rightarrow A$ is a function.  
 
For the family of isomorphisms of monoidal functors
$\theta^{\sigma,\tau}=(\theta^{\sigma,\tau}_x)$, we are able to write\[
\theta^{\sigma,\tau}_x=(\varphi^{\sigma\tau}x, t^{\sigma,\tau}(x))=
(\varphi^{\sigma\tau}x, t(x,\sigma,\tau)) \] where $t: \Pi \times \Gamma^2 \rightarrow A$ is a function.

From functions $\xi,\widetilde{f}, t$ , in which $\xi$ is associated with the associative
constraint $a$ of $(\Pi,A)$, we determine the function $h$ as follows: \[ h:\Pi^3 \cup
(\Pi^2 \times \Gamma) \cup (\Pi \times \Gamma^2) \rightarrow A \] 
where $h= \xi\cup \widetilde{f}\cup t,$ in the sense
\[h \mid_{\Pi^3}=\xi;\quad h \mid_{\Pi^2 \times \Gamma}=\widetilde{f}; \textrm{ and } h \mid_{\Pi \times \Gamma^2}=t\]
The above determined $h$ is a $\Gamma-$operator $3-$cocycle. Indeed, the equalities $(1), (5)$ turn into:
\begin{equation} -\sigma x.h(y,z,\sigma)+h(xy,z,\sigma)+h(x,y,\sigma)-h(x,yz,\sigma)
=h(\sigma x,\sigma y,\sigma z)- \sigma(h(x,y,z))
\end{equation}
\begin{equation}  (\sigma\tau)x.h(y,\sigma,\tau)-h(xy,\sigma,\tau)-h(x,\sigma,\tau)
= h(x,y,\sigma\tau)-h(\tau x, \tau x,\sigma)-\sigma
h(x,y,\tau)\end{equation} 
Moreover, from the relations of $3-$cocycle
$\xi,$ we obtain:
\begin{equation}
xh(y,z,t)-h(xy,z,t)+h(x,yz,t)-h(x,y,zt)+h(x,y,z)=0
\end{equation}
The cocycle condition\[
\theta^{\sigma\tau,\gamma}.\theta^{\sigma,\tau}F^{\gamma}=
\theta^{\sigma,\tau\gamma}.F^{\sigma}\theta^{\tau,\gamma} \] 
yeilds
\begin{equation}
h(x,\sigma\tau,\gamma)+h(\gamma x,\sigma,\tau)= h(x,\sigma,
\tau\gamma)+\sigma h(x,\tau,\gamma)
\end{equation} for all $x,y,z \in \Pi; \quad \sigma,\tau \in \Gamma$. It follows that $h$ satisfies the relations of a $3-$cocycle in $Z^3_{\Gamma}(\Pi,A)$. However, we have to prove the normalized property of $h$.

First, since the unit constraints of $(\Pi, A)$ are strict and the factor set $(\theta,F)$ is enough strict, the equalities $(2), (3), (6)$ turn into:
$$ h(x,1,\sigma)= \widetilde{f}^{\sigma}(x, 1)=0=\widetilde{f}^{\sigma}(1, x)=h(1,x,\sigma)$$
$$ h(1,\sigma,\tau)=t^{\sigma, \tau}(1)= 0 $$
Since $\widehat{F^1}=id,$ we have
$$ h(x,y,1_{\Gamma})=\widetilde{f}^1(x,y)=0$$
Since the normalized property of associative constraint $a,$ 
$$ h(1,y,z)= h(x,1,z)=h(x,y,1)=0$$
Thanks to ii) in the definition of a factor set, we have
$$h(x,1_{\Gamma},\tau)=h(x,\sigma,1_{\Gamma})=0$$

\noindent Let $h^{(\theta,F)}=h,$ we have $h^{(\theta,F)} \in Z^3_{\Gamma}(\Pi,A)$. This completes the proof of theorem.
\end{proof} 
\section{Classification theorem}

Let $(\mu,G)$ be another enough strict factor set on $\Gamma$ with coefficients in $((\Pi,A,\xi),$ which is cohomologous to $(\theta,F)$. Thus, $\Pi-$module $\Gamma-$equivariant structure $A$ and the element $h^{(\mu, G)}\in H^3_{\Gamma}(\Pi,A)$ defined by
$(\mu,G)$ has the following property: 

\begin{md} Let two enough strict factor sets $(\mu, G)$, $(\theta, F)$ on $\Gamma$ with coefficients in categorical group $\mathcal S$ of the type $(\Pi,A)$ be cohomologous. Then, they determine the same structure of $\Pi-$module $\Gamma-$equivariant on $A$ and $3-$cocyles inducing $h^{(\theta,F)}$, $h^{(\mu, G)}$ are cohomologous.
\end{md}
\begin{proof}
According to Remark \ref{nx1}, $F^{\sigma}=G^{\sigma} (\sigma \in \Gamma$. Then, they induce the same $\Pi$-module $\Gamma$-equivariant structure, according to the relation (7) and Theorem \ref{dl2.3}.

Now, we prove that elements $h^{(\theta,F)}$ and $h^{(\mu, G)}$ are cohomologous.
We denote $h^{(\mu, G)}=h'$. Hence, by determining of $h^{(\mu,G)}$ referred in Proposition 3.4, we have
\[  (\sigma(xy),h'(x,y,\sigma))= \widetilde{H^\sigma}_{x,y};
(\sigma \tau x,h'(x,\sigma,\tau)= \mu^{\sigma,\tau}x,\]
for all  $x,y \in \Pi,\sigma, \tau \in \Gamma$ 
Let $u: \Pi \times \Gamma \rightarrow A$ be the function defined by $u(x, \sigma)=u^{\sigma}_{x}$. It determines an extending $2-$cochain of $u$, denoted by $g$, with $g|_{\Pi^2}: \Pi^2 \rightarrow A$ is the null map.\\
\indent Since $(u^{\sigma\tau}.\theta^{\sigma,\tau})x =
(\mu^{\sigma,\tau}. u^{\sigma}G^{\tau} .
F^{\sigma}u^{\tau})x,$ we have
\begin{equation}
g(x,\sigma\tau)+h(x,\sigma,\tau)=h'(x,\sigma,\tau)+g(\tau x,
\sigma)+\sigma g(x,\tau)
\end{equation}

Since $ \widetilde{G^\sigma}_{x,y}\cdot u^{\sigma}_ {x \otimes y} =
 (u^{\sigma}_{x} \otimes
u^{\sigma}_{y})\cdot \widetilde{F^\sigma}_{x,y},$ we have
\begin{equation}
h'(x,y,\sigma)-h(x,y,\sigma)= g(x,\sigma)+\sigma
 xg(y,\sigma)-g(xy,\sigma)
\end{equation}

Since $\widehat{F^\sigma} = \widehat{H^\sigma}= id,$ we have $u^{\sigma}_1
= \widehat{H^\sigma}(\widehat{F^\sigma})^{-1}=id_1$. Hence
\begin{equation} g(1,\sigma)=0 \quad \text{ for all } \sigma \in \Gamma
\end{equation}

Since $u^1 = id_{((\Pi,A),\otimes)},$ we have
\begin{equation} g(x,1_{\Gamma})=0 \quad \text{ for all } x
\in G \end{equation}

By the determining of $g$ and relations (21)-(24), we have $g \in C^2_{\Gamma}(\Pi,A)$ and
$h^{(\theta,F)}-h^{(\mu,G)}= \partial g$. This completes the proof of proposition.
\end{proof} 

Thus, any factor set $(\theta,F)$ on $\Gamma$ with coefficients in categorical groups of the type $(\Pi,A)$ determines a structure of $G-$module $\Gamma-$equivariant $A$ and an element $h^{(\theta,F)} \in
H^3_{\Gamma}(\Pi,A)$ uniquely.

Now, we consider the problem: Giving an element $h \in
Z^3_{\Gamma}(\Pi,A),$ does there exist a factor set $(\theta,F)$ on $\Gamma$ with coefficients in $(\Pi,A)$, inducing $h$.

According to the definition of $\Gamma-$operator cocycle, we have functions
\[h \mid_{\Pi^3}=\xi;\quad h \mid_{\Pi^2 \times \Gamma}=\widetilde{f}; \text{ and } h \mid_{\Pi \times \Gamma^2}=t\]
Hence, we may determine the factor set $(\theta, F)$:
\[F^{\sigma}x =
\sigma x; \quad F^{\sigma}(x,c)=(\sigma x, \sigma c)  \quad \text{ that is } f^{\sigma}(c)= \sigma c \] 
\[\widehat{F^\sigma}=id_1, \quad
 \widetilde{F^\sigma}_{x,y} = (\sigma(xy), \widetilde{f}(x,y,\sigma)) \] 
 \[\theta^{\sigma\tau}x = (\sigma\tau x, t(x,\sigma,\tau)),\]
for all $\sigma, \tau \in \Gamma, x \in \Pi, c \in A.$ Clearly, the above determined factor set induces $h$.

We now state the main result of the paper:

\begin{dl} There exists a bijection
$$\Omega: \mathcal S_{\Gamma}(\Pi, A)\rightarrow H^3_{\Gamma}(\Pi,A),$$
where $\mathcal S_{\Gamma}(\Pi, A)$ is the set of congruence classes of $\Gamma-$extensions of categorical groups of the type $(\Pi,A)$.
\end{dl}

\begin{proof}
Any element of $\mathcal S_{\Gamma}(\Pi, A)$ may have a crossed product $\Gamma-$extension $\Delta ((\theta, F), \C)$ be a presentative, where $\C$ is a categorical group $(\Pi, A, \xi).$ According to Proposition 3.3, it is possible to assume that $(\theta, F)$ is enough strict. Then, $(\theta, F)$ induces $3-$cocycle $h=h^{(\theta, F)}$ (Proposition 3.4). According to Proposition 4.1, the correspondence $cl(\theta, F)\rightarrow cl(h^{(\theta, F)})$ is a map. Thanks to the above remark, this correspondence is a surjection. We now prove that it is an injection.

Let $\Delta$ and $\Delta'$ be crossed product $\Gamma-$extension of Gr-categories $\mathcal S=\mathcal S(\Pi, A,\xi)$, $\mathcal S'=\mathcal S'(\Pi, A,\xi')$ by factor sets $(\theta, F)$, $(\theta', F')$. Moreover, $3-$cocycles inducing $h, h'$ are cohomologous. We will prove that $\Delta$ and $\Delta'$ are equivariant $\Gamma-$extensions.

According to the determining of $h, h',$ 
$$h \mid_{\Pi^3}=\xi,\ \ h' \mid_{\Pi^3}=\xi', \ \ \xi'=\xi+\delta g,$$ 
where $g:\Pi^2\rightarrow A $ is a map.
Then, Gr-categories $\mathcal S$ and $\mathcal S'$ are Gr-equivalent, so there exists a Gr-equivalence
$$(K, \widetilde{K}): \mathcal S\rightarrow \mathcal S',$$
where $\widetilde{K}_{x, y}=(\bullet, g(x, y)).$ We may extend the Gr-functor $(K, \widetilde{K})$ to a $\Gamma-$functor
$$(K_\Delta, \widetilde{K_\Delta}):\Delta\rightarrow \Delta'$$ 
as follows:
$$K_\Delta(x)=K(x)\ ; K_\Delta(a, \sigma)=(K(a), \sigma)\ ; \widetilde{K_\Delta}=\widetilde{K}.$$
It is easy to see that $(K_\Delta, \widetilde{K_\Delta})$ is a $\Gamma-$equivalence. Hence,
$\Omega$ is an injection.
\end{proof}
\vspace{0.5cm}

\begin{center}

\end{center}
 Address: Department of Mathematics\\
Hanoi National University of Education\\
136 Xuan Thuy Street, Cau Giay district, Hanoi, Vietnam.\\
Email: nguyenquang272002@gmail.com

\end{document}